\documentclass[reqno]{amsart}
\usepackage{amsmath}
\usepackage{amssymb, hyperref}
\newcommand{\R}{\mathbb{R}}
\newcommand{\C}{\mathbb{C}}

\usepackage{graphicx,color}
\newcommand{\Real}{\mathbb R}
\begin{document}
\title[On the IBNLS equation]{Some remarks on the inhomogeneous biharmonic NLS equation}
	\author[C. M. GUZM\'AN ]
	{CARLOS M. GUZM\'AN }  
	
	\address{CARLOS M. GUZM\'AN \hfill\break
		Department of Mathematics, Fluminense Federal University, BRAZIL}
	\email{carlos.guz.j@gmail.com}
	
    \author[A. PASTOR ]
	{ADEMIR PASTOR}  
	
	\address{ADEMIR PASTOR  \hfill\break
	Imecc-Unicamp,	Rua S\'ergio Buarque de Holanda, 651, 13083-859, Campinas--SP, Brazil.}
	\email{apastor@ime.unicamp.br}

\begin{abstract}
We consider the inhomogeneous biharmonic nonlinear Schr\"odinger equation 
$$
i u_t +\Delta^2 u+\lambda|x|^{-b}|u|^\alpha u = 0, 
$$
where $\lambda=\pm 1$ and $\alpha$, $b>0$. In the subctritical case, we improve the global well-posedness result obtained in \cite{GUZPAS} for dimensions $N=5,6,7$ in the Sobolev space $H^2(\mathbb{R}^N)$. The fundamental tools to establish our results are the standard Strichartz estimates related to the linear problem and the Hardy-Littlewood inequality. Results concerning the energy-critical case, that is, $\alpha=\frac{8-2b}{N-4}$ are also reported. More precisely, we show well-posedness and a
stability result with initial data in the critical space $\dot{H}^2$.
\end{abstract}

	
\maketitle  
\numberwithin{equation}{section}
	\newtheorem{theorem}{Theorem}[section]
	\newtheorem{proposition}[theorem]{Proposition}
	\newtheorem{lemma}[theorem]{Lemma}
	\newtheorem{corollary}[theorem]{Corollary}
	\newtheorem{remark}[theorem]{Remark}
	\newtheorem{definition}[theorem]{Definition}
	
\section{\bf Introduction}

This paper deals with the initial-value problem (IVP) associated to the inhomogeneous biharmonic nonlinear Schr\"odinger equation (IBNLS for short)
\begin{equation}\label{IBNLS}
\begin{cases}
i\partial_tu +\Delta^2 u + \lambda|x|^{-b} |u|^\alpha u =0,  \;\;\;t\in \mathbb{R} ,\;x\in \mathbb{R}^N,\\
u(0,x)=u_0(x), 
\end{cases}
\end{equation}
where $\Delta^2$ denotes the biharmonic operator, $u = u(t,x)$ is a complex-valued function in space-time  $\mathbb{R}\times\mathbb{R}^N$, $\alpha, b>0$ and $\lambda$ is a real constant that may be normalized to  $\pm 1$. The equation in \eqref{IBNLS} is an inhomogeneous version of the well known  biharmonic nonlinear Schr\"odinger equation (BNLS) (also termed fourth-order Schr\"odinger equation), see \cite{FIBI1}. Moreover, the equation is called ``focusing IBNLS" when $\lambda= -1$ and ``defocusing IBNLS" when $\lambda= 1$. 

To the best of our knowledge, from the mathematical viewpoint, the study of well-posedness for fourth-order Schr\"odinger-type equations with inhomogeneous nonlinearities was initiated in \cite{Cho-Ozawa}, where the authors considered the inhomogeneous power $|x|^{-2}|u|^{\frac{4}{N}}u$ and established  results of local well-posedness, and blow up in finite time for initial data with negative energy and radial symmetry. Concerning more general power-like nonlinearities of the form $|x|^{-b}|u|^{\alpha}u$, very recently,  in \cite{GUZPAS}, the authors proved several well-posedness results in $L^2(\mathbb{R}^N)$ and $H^2(\mathbb{R}^N)$. Local well-posedness in the energy space has also appeared in \cite{liuzhang2021}. Global well-posedness for large initial data in the energy space and concentration in the critical Lebesgue space were proved in \cite{carguzpas2020}.  In \cite{saanouni2021} the author established a scattering result in the energy space.

Before describing our result we first observe that  the IBNLS equation is invariant by scaling in the following sense: if $u$ is a solution of \eqref{IBNLS} then the scaled function 
$u_\mu(t,x)=\mu^{\frac{4-b}{\alpha}}u(\mu^4 t,\mu x)$,  $\mu >0$, is also a solution of \eqref{IBNLS} with initial data $u_{0,\mu}(x)=\mu^{\frac{4-b}{\alpha}}u_0(\mu x)$. In addition, it is easily seen that
$$
\|u_{0,\mu}\|_{\dot{H}^s}=\mu^{s-\frac{N}{2}+\frac{4-b}{\alpha}}\|u_0\|_{\dot{H}^s},
$$
which means that the Sobolev space invariant by the above scaling is $\dot{H}^{s_c}(\mathbb{R}^N)$, with $s_c=\frac{N}{2}-\frac{4-b}{\alpha}$. Thus, as usually termed in dispersive equations, the case $s_c = 0$ (equivalently $\alpha = \frac{8-2b}{N}$) is said to be mass-critical or $L^2$-critical since at this level the equation preserves the mass given by
\begin{equation*}
M[u(t)]=\int_{\mathbb{R}^N}|u(t,x)|^2dx.
\end{equation*}
Also, the case
$s_c=2$ (equivalently $\alpha =\frac{8-2b}{N-4}$) is said to be energy-critical or $\dot{H}^2$-critical because the equation conserves the energy
\begin{equation*}
E[u(t)]=\frac{1}{2}\int_{\mathbb{R}^N}| \Delta u(t,x)|^2dx+\frac{\lambda}{\alpha +2} \int_{\mathbb{R}^N} |x|^{-b}|u|^{\alpha +2}dx.
\end{equation*}
Finally, in the case $0<s_c<2$ the equation is said to be mass-supercritical and energy-subcritical (or intercritical). This case is equivalent to  $\frac{8-2b}{N}<\alpha<4^*$, where 
\begin{equation}\label{def4*}
 4^*:=
 \begin{cases}
 \frac{8-2b}{N-4},\;\;\mbox{if}\;\;N\geq 5,\\
 +\infty,\;\;\mbox{if}\;\;1\leq N\leq 4.
 \end{cases}
\end{equation}

Now we are able to describe our results and to motivate the present work. In \cite{GUZPAS}, for $N\geq 3$, we showed local well-posedness in $H^2(\mathbb{R}^N)$ with $0<b<\min\{\frac{N}{2},4\}$ and $\min\left\{\frac{2(1-b)}{N},0\right\}<\alpha<4^*$. 
Moreover, we also established global well-posedness in the mass-subcritical and mass-critical cases, that is, $\min\left\{\frac{2(1-b)}{N},0\right\}<\alpha\leq {\frac{8-2b}{N}}$. These global results were proved as an application of a Gagliardo-Nirenberg-type inequality,  taking account the conservation of the mass and the energy. In the intercritical case, we proved the small data global existence under some restrictions on the parameters $b$, $\alpha$ and $N$. More specifically, the following result was proved.

\vskip.2cm

\noindent {\bf Theorem A.} {\it
Assume one of the following conditions:
\begin{itemize}
\item[(i)] $N\geq8$, $0<b<4$, and $\frac{8-2b}{N}<\alpha<4^*$;
\item[(ii)] $N=5,6,7$, $\frac{8-2b}{N}<\alpha<\frac{N-2b}{N-4}$ and $0<b<\frac{N^2-8N+32}{8}$;
\item[(iii)] $N=6,7$, $0<b<N-4$, and	$\frac{8-2b}{N}<\alpha<4^*$;
\item[(iv)] $N=3,4$, $0<b<\frac{N}{2}$, and $\frac{8-2b}{N}<\alpha<\infty$.
\end{itemize}	
Suppose $u_0 \in H^2(\mathbb{R}^N)$ satisfies $\|u_0\|_{H^2}\leq \eta$, for some $\eta>0$. Then there exists $\delta=\delta(\eta)>0$ such that if $\|e^{it\Delta^2}u_0\|_{B(\dot{H}^{s_c})}<\delta$, then  there exists a unique global solution $u$ of \eqref{IBNLS}.
}

\vskip.2cm

Note that in dimensions $N=5,6,7$, Theorem A does not cover the range for $\alpha$ and $b$ where the local well-posedness was established, that is,  $\frac{8-2b}{N}<\alpha<4^*$ and  $0<b<\min\{\frac{N}{2},4\}$. So, in this paper, our first interest is to improve these global results by extending the range of the parameter $\alpha$. Our first main result the following.

\begin{theorem}\label{GWPH2}
Assume $N\geq3$, $0<b<\min\{\frac{N}{2},4\}$, and $\frac{8-2b}{N}<\alpha<4^*$ $(\frac{8-2b}{N}<\alpha<7-2b$ if $N=5)$. If $u_0 \in H^2(\mathbb{R}^N)$ satisfies $\|u_0\|_{H^2}\leq \eta$, for some $\eta>0$. Then there exists $\delta=\delta(\eta)>0$ such that if $\|e^{it\Delta^2}u_0\|_{B(\dot{H}^{s_c})}<\delta$, then  there exists a unique global solution $u$ of \eqref{IBNLS} such that
\begin{equation*}*\label{NGWP3}
\|u\|_{B(\dot{H}^{s_c})}\leq  2\|e^{it\Delta^2}u_0\|_{B(\dot{H}^{s_c})}\qquad \textnormal{and}\qquad \|u\|_{B\left(L^2\right)}+\|\Delta  u\|_{B\left(L^2\right)}\leq 2c\|u_0\|_{H^2}.
\end{equation*}
for some universal constant $c>0$.
\end{theorem}

Observe that in dimension $N=5$ we have the extra  restriction   $\alpha<7-2b$ (instead of $\alpha<8-2b$), which still leaves a lack in the range of $\alpha$. This restriction seems to be technical and appears in the proof of our nonlinear estimates (see Lemma \ref{improvedlemma}). However, if we insist with the assumption $\frac{8-2b}{5}<\alpha<8-2b$, then we need to restrict the range of $b$.

\begin{theorem}\label{GWPH2'}
Let $N=5$, $0<b\leq\frac{3}{2}$ and $\frac{8-2b}{5}<\alpha<8-2b$. If $u_0 \in H^2(\mathbb{R}^N)$ satisfies $\|u_0\|_{H^2}\leq \eta$, for some $\eta>0$. Then there exists $\delta=\delta(\eta)>0$ such that if $\|e^{it\Delta^2}u_0\|_{L^{a^*}_tL^{r^*}_x}<\delta$, then  there exists a unique global solution $u$ of \eqref{IBNLS} such that
\begin{equation*}
\|u\|_{L^{a^*}_tL^{r^*}_x}\leq  2\|e^{it\Delta^2}u_0\|_{L^{a^*}_tL^{r^*}_x}\quad \textnormal{and}\quad \|u\|_{B\left(L^2\right)}+\|\Delta  u\|_{B\left(L^2\right)}\leq 2c\|u_0\|_{H^2},
\end{equation*}
for some universal constant $c>0$,  where
\begin{align}\label{pairN=5a}
a^*\;=\;\frac{8\alpha(\alpha+1-\theta)}{8-2b-\alpha(N-4)+2\varepsilon\alpha}\quad\textnormal{and}\quad r^*=\frac{2\alpha N(\alpha+1-\theta)}{\alpha(N+4-2b)-2\theta (4-b)-2\varepsilon \alpha},
\end{align}
with $\theta$ and $\varepsilon$ sufficiently small numbers.
\end{theorem}

The restriction $b\leq\frac{3}{2}$ comes from the fact that wee need $\alpha > 1$ to prove our nonlinear estimates (see Lemma \ref{LG3}), which forces $\frac{8-2b}{5}\geq1$ or, equivalently, $b\leq\frac{3}{2}$. Of course, if $\alpha>1$ then we do not need such a restriction and we may establish the following.

\begin{corollary}
If $N=5$, $0<b<\frac{5}{2}$ and  $\min\{1,\frac{8-2b}{5}\} <  \alpha< 8-2b$ then the same conclusion of Theorem \ref{GWPH2'} holds.
\end{corollary}

Theorems \ref{GWPH2} and \ref{GWPH2'} improve the results presented in Theorem A for $N=5,6,7$. The main new tool used to prove our nonlinear estimates is  the Hardy-Littlewood inequality (see Lemma \ref{Hardy}). As usual,  based on the Strichartz estimates, we construct a suitable metric space  such that the integral operator 
\begin{equation}\label{OPERATOR} 
G(u)(t)=e^{it\Delta^2}u_0+i\lambda \int_0^t e^{i(t-t')\Delta^2}|x|^{-b}|u(t')|^\alpha u(t')dt'
\end{equation}
is a contraction, which in view of the contraction mapping principle gives our global solution. Here and throughout the paper,  $e^{it\Delta^2}u_0$ stands for the solution of the linear problem associated with \eqref{IBNLS}. 

We point out that, for $N\geq3$ and $N\neq 5$, we obtain the best possible results, in the sense that $\alpha$ and $b$ range in the largest possible intervals, i.e., $\alpha$ involves the whole intercritical range and $b$ the same range where the local theory was proved (see \cite[Theorem 1.2]{GUZPAS}). On the other hand, in dimension $N=5$, when $\frac{8-2b}{5}<\alpha < 8-2b$ the case $\frac{3}{2}<b<\frac{5}{2}$ is not covered by our results and it is left as an open problem.

Next, making use of the estimates on $|x|^{-b}|u|^\alpha u$, we may also improve the scattering criterion and the stability
results shown in \cite{GUZPAS}, in the intercritical case. More precisely, we may prove the following.

\begin{proposition}\label{SCATTERSH1}  Let $u_0 \in  H^2(\mathbb{R}^N)$ and let $u$ be the corresponding global solution of \eqref{IBNLS}. Suppose  $\|u\|_{B(\dot{H}^{s_c})}< +\infty$ (or $\|u\|_{L^{a^*}_tL^{r^*}_x}< +\infty$) and $\sup\limits_{t\in \mathbb{R}}\|u(t)\|_{H^2_x}\leq \eta$. If the assumptions in Theorem \ref{GWPH2} (or Theorem \ref{GWPH2'}) hold, then $u$ scatters in $H^2(\mathbb{R}^N)$ as $t \rightarrow \pm\infty$. 
\end{proposition}

\begin{proposition}\label{LTP} 
Let $I\subseteq \mathbb{R}$ be an interval containing zero and 
suppose that assumptions in Theorem \ref{GWPH2} hold. Assume  $\widetilde{u}_0\in H^2(\mathbb{R}^N)$ and let  $\widetilde{u}$  be a solution of
\begin{equation*}\label{appsol}
\begin{cases}
i\partial_t \widetilde{u} +\Delta^2 \widetilde{u} + \lambda |x|^{-b} |\widetilde{u}|^\alpha \widetilde{u} =e,\quad t\in I, x\in\mathbb{R}^N,\\
\widetilde{u}(0)=\widetilde{u}_0,
\end{cases}
\end{equation*}  
 satisfying 
\begin{equation*}\label{HLP1} 
\sup_{t\in I}  \|\widetilde{u}\|_{H^2_x}\leq M \;\; \textnormal{and}\;\; \|\widetilde{u}\|_{B(\dot{H}^{s_c}; I)}\leq L,
\end{equation*}
 where $M$ and $L$ are positive constants.
Let $u_0\in H^2(\mathbb{R}^N)$ be such that 
\begin{equation*}\label{HLP2}
\|u_0-\widetilde{u}_0\|_{H^2}\leq M'\;\; \textnormal{and}\;\; \|e^{it\Delta^2}(u_0-\widetilde{u}_0)\|_{B(\dot{H}^{s_c}; I)}\leq \varepsilon,
\end{equation*}
for some positive constant $M'$ and some $0<\varepsilon<\varepsilon_1$. Moreover, assume also the following error estimates
\begin{equation*}
\|e\|_{B'(\dot{H}^{s_c}; I)}+\| e\|_{B'(L^2; I)}+ \|\nabla e\|_{L^2_IL^{\frac{2N}{N+2}}_x}\leq \varepsilon.
\end{equation*}
\indent Then, there exists a unique solution $u$ to \eqref{IBNLS} defined on $I\times \mathbb{R}^N$, with initial data   $u_0$ such that
\begin{equation*}\label{CLP} 
\|u-\widetilde{u}\|_{B(\dot{H}^{s_c}; I)}\lesssim\varepsilon\quad\textnormal{and} \quad \|u\|_{B(\dot{H}^{s_c}; I)} +\|u\|_{B(L^2; I)}+\|\Delta u\|_{B(L^2; I)}\lesssim 1,
\end{equation*}
where the implicit constants depend on $M, M'$ and $L$.
\end{proposition}

Note that in Proposition \ref{LTP} we assume only that the assumption in Theorem \ref{GWPH2} hold (we do not consider the conditions of Theorem \ref{GWPH2'}). The reason is that to prove our nonlinear estimates we need to estimate the quantities $\nabla (|x|^{-b}|u|^\alpha u)$  in Theorem \ref{GWPH2} and $\Delta (|x|^{-b}|u|^\alpha u)$ in Theorem \ref{GWPH2'}. When dealing with the gradient we have nice estimates that allows to prove Proposition \ref{LTP}. On the other hand, the estimates on the Laplacian operator are not suitable enough to establish the same result. 

As a second interest in this paper, we study the IVP associated to the energy-critical IBNLS, that is, \eqref{IBNLS} with $\alpha =\frac{8-2b}{N-4}$. Before stating the results, let us recall some results for \eqref{IBNLS} in the limiting case $b = 0$, which is is termed as the biharmonic nonlinear Schr\"odinger equation (BNLS). In the defocusing case ($\lambda=1$), by combining the concentration-compactness argument (see \cite{KENIG}) with some Morawetz-type estimates, in \cite{Pausader07}, the author showed global well-posedness and scattering for $\alpha =\frac{8}{N-4}$ assuming radially symmetric initial data.  Later, in \cite{Miao-Xu-Zhao} the authors showed a similar result removing the radial assumption on the initial data, for $N \geq 9$. In \cite{Pausader-cubica} was showed the global well-posedness and scattering for the cubic BNLS  ($\alpha=2$) and $N=8$. In the focusing case ($\lambda=-1$), in \cite{Pausader09} and  \cite{Miao-Xu-Zhao09} the authors independently showed the global well-posedness and scattering in the energy-critical case, with radial data in $\dot{H}^2(\mathbb{R}^N)$ and energy norms below that of the ground states.

Our main  interest here is to study the problem of well-posedness and stability for \eqref{IBNLS}. First, we show that IBNLS is  well-posed in $\dot{H}^2(\Real^N)$ with some restrictions on the dimension.

\begin{theorem}\label{GWPCH2}
Assume $5\leq N\leq 11$, $\alpha=\frac{8-2b}{N-4}$ and $0<b<\frac{12-N}{N-2}$. Let $I\subseteq \mathbb{R}$ be a time interval containing zero. If $u_0 \in \dot{H}^2(\mathbb{R}^N)$ satisfies $\|u_0\|_{\dot{H}^2}\leq A$, for some $A>0$. Then there exists $\delta=\delta(A)>0$ such that if $\|e^{it\Delta^2}u_0\|_{B(I)}<\delta$, then  there exists a unique solution $u$ of \eqref{IBNLS} in $I\times \Real^N$ such that
\begin{equation*}
\|u\|_{B(I)}\leq  2 \delta \qquad \textnormal{and}\qquad \|\Delta u\|_{B(L^2;I)}\leq 2c A,
\end{equation*}
for some universal constant $c>0$, where
\begin{equation}\label{BI}
B(I)=L_I^{\frac{2(N+4)}{N-4}}L_x^{\frac{2(N+4)}{N-4}}.  
\end{equation} 
\end{theorem}

To show Theorem \ref{GWPCH2}, we need to establish a good estimate on the nonlinearity and to this end we use the Hardy-Littlewood inequality. The advantage here is that, contrary to estimates in the intercritical case, we do not need to divide the estimates inside and outside the unit ball. Note that in the previous theorem we have imposed the restriction $b<\frac{12-N}{N-2}$; this comes from the fact that to apply the fixed point argument we need $\alpha-b-1>0$ (see Lemma \ref{LemCritico}). Note that this also forces $N\leq 11$. We believe it would be interesting to extend the range of the parameter $b$ and establish similar results in higher dimensions, $N\geq 12$.

Finally, we  show a stability result for the solutions of \eqref{IBNLS} in the critical case.

\begin{theorem}\label{LTPC} 
Assume that assumptions in Theorems \ref{GWPCH2} hold. For a given interval $I$ with $0\in I\subseteq \mathbb{R}$, let $\tilde u:I\times \mathbb{R}^N\to\C$ be a solution of
\begin{equation}\label{erroreq}
i\partial_t \tilde{u} +\Delta^2 \tilde{u} + \lambda|x|^{-b} |\tilde{u}|^\alpha \tilde{u} =e,
\end{equation}  
with initial data $\tilde{u}_0\in \dot{H}^2(\mathbb{R}^N)$ satisfying 
\begin{equation}\label{erroreq1} 
\sup_{t\in I}  \|\tilde{u}\|_{\dot{H}^2_x}\leq M \;\;\; \textnormal{and}\;\;\; \|\tilde{u}\|_{B(I)}\leq L,\;\;\; \textnormal{for}\;\;M,L>0.
\end{equation}
Assume that $u_0\in \dot{H}^2(\mathbb{R}^N)$ satisfies
\begin{equation}\label{erroreq2}
\|u_0-\tilde{u}_0\|_{\dot{H}^2}\leq M',\quad \|\nabla e\|_{L^2_IL_x^{\tfrac{2N}{N+2}}}\leq \varepsilon  \quad{and}\quad \|e^{it\Delta^2}(u_0-\tilde{u}_0)\|_{B(I)}\leq \varepsilon,
\end{equation}
 for $M'>0$ and some $0<\varepsilon<\varepsilon_1$.

\indent Then, there exists a unique solution $u:I\times\R^N\to\C$ with $u(0)=u_0$ obeying	
\begin{align*}
\|u-\tilde{u}\|_{B(I)}\lesssim\varepsilon\quad \textnormal{and}\quad
\|u\|_{B(I)} +\|\nabla u\|_{B(L^2; I)} \lesssim 1,
\end{align*} 
where the implicit constants depend on $M, M'$ and $L$.
\end{theorem}

\ The rest of the paper is organized as follows. In section \ref{Sec2}, we introduce some notations and give a review of the Strichartz estimates. In Section \ref{Sec3}, we prove the global well-posedness and stability results in the intercritical case. In the last section, Section $4$, we study the energy critical IBNLS by proving Theorems \ref{GWPCH2} and \ref{LTPC}.

\ 		

\section{\bf Preliminaries}\label{Sec2}

Let us start this section by introducing some notations.  By $c$ we denote various constants that may vary line by line. When the value of the constant $c$ is irrelevant in the estimates, we use the
notation $a \lesssim b$ to say that there exists a constant $c>0$ such that $a \leq cb$. 

By $L^p=L^p(X)$ we denote the standard Lebesgue space on the measurable space $X$ with the usual norm $\|\cdot\|_{L^p}$. If $f=f(t,x)$ is a function of the variable $t$ and $x$, we use $\|f\|_{L^p_x}$ to indicate that we are taking the $L^p$ norm with respect to the variable $x$. Also, given a time interval $I\subset \mathbb{R}$ we use $L^p_I$ to denote the $L^p$ space on $I$. In the case $I=\mathbb{R}$ we shall use $L^p_t$ instead of $L^p_{\mathbb{R}}$.  Given two positive numbers $q$ and $r$,  the  norm in the mixed   space $L^q_{I}L^r_x$ is given by
$$
\|f\|_{L^q_{I}L^r_x}=\left\|\|f(t,\cdot)\|_{L^r_x}\right\|_{L^q_I}.
$$

Let  $J^s$ and $D^s$ denote the Bessel and Riesz potentials of order $s$, that is, via Fourier transform, $\widehat{J^s f}=(1+|\xi|^2)^{\frac{s}{2}}\widehat{f}$ and $\widehat{D^sf}=|\xi|^s\widehat{f}.$ The norm in the Sobolev spaces $H^{s,r}=H^{s,r}(\mathbb{R}^N)$ and $\dot{H}^{s,r}=\dot{H}^{s,r}(\mathbb{R}^N)$ are given by $\|f\|_{H^{s,r}}:=\|J^sf\|_{L^r}$ and $\|f\|_{\dot{H}^{s,r}}:=\|D^sf\|_{L^r},$ respectively.
In the particular case $r=2$, as usual, we denote $H^{s,2}$ and $\dot{H}^{s,2}$  by $H^s$ and  $\dot{H}^{s}$, respectively. Moreover, if $s=0$ then $H^0=\dot{H}^{0}=L^2$.

We now recall two important inequalities.

\begin{lemma}[\textbf{Hardy-Littlewood inequality}]\label{Hardy}
For $1 < p \leq q < +\infty$, $N \geq 1$, $0 < s < N$ and $\rho \geq 0$ suppose that
$$
\rho < \frac{N}{q} \quad \mbox{and} \quad s =\frac{N}{p}-\frac{N}{q}+\rho.
$$
Then, for any $u \in H^{s,p}(\mathbb{R}^N)$ we have
$$
\||x|^{-\rho}u\|_{L^q}
\lesssim \|D^su \|_{L^p}.
$$
\end{lemma}
\begin{proof}
See Theorem B* in \cite{SteinW}.
\end{proof}

\begin{lemma}[\textbf{Fractional Gagliardo-Nirenberg inequality}]\label{GNinequality}
Assume $1<p,p_0,p_1<\infty$,\newline $s,s_1\in\mathbb{R}$, and $\theta\in[0,1]$. Then the fractional Gagliardo-Nirenberg inequality
$$
\| D^s u \|_{L^p} \lesssim \|u\|_{L^{p_0}}^{1-\theta}\|D^{s_1}u\|_{L^{p_1}}^\theta.
$$
holds if and only if
$$
\frac{N}{p}-s=(1-\theta)\frac{N}{p_0}+\theta\left(\frac{N}{p_1}-s_1\right), \quad s\leq \theta s_1.
$$
\end{lemma}
\begin{proof} 
See Corollary 1.3 in \cite{Wangatall}.
\end{proof}

\ Next, we recall some Strichartz type estimates associated to the linear biharmonic Schr\"odinger equation. Given a real number $s<2$, the pair $(q,r)$ is called $\dot{H}^s$-biharmonic admissible  if 
\begin{equation}\label{CPA1}
\frac{4}{q}=\frac{N}{2}-\frac{N}{r}-s
\end{equation}
with
\begin{equation}\label{HsAdmissivel}
\begin{cases}
\frac{2N}{N-2s} \leq  r  <\frac{2N}{N-4},\;\;\textnormal{if}\;\;  N\geq 5,\\
2 \leq  r < + \infty,\;\;\;  \hspace{0.65cm}\textnormal{if}\;\;\;1\leq N\leq 4.
\end{cases}
\end{equation}

In the particular case $s=0$ we shall say that  $(q,r)$ is biharmonic Schr\"odinger admissible (or $B$-admissible for short). By setting $\mathcal{B}_s:=\{(q,r);\; (q,r)\; \textnormal{is} \;\dot{H}^s\textnormal{-biharmonic admissible}\}$, we define the Strichartz norm by
$$
\|u\|_{B(\dot{H}^{s})}=\sup_{(q,r)\in \mathcal{B}_{s}}\|u\|_{L^q_tL^r_x} 
$$
and its dual norm by
$$
\|u\|_{B'(\dot{H}^{-s})}=\inf_{(q,r)\in \mathcal{B}_{-s}}\|u\|_{L^{q'}_tL^{r'}_x},
$$
where $p'$ indicates the H\"older conjugate of $p$.
When the time norm is restricted to some interval $I\subset\mathbb{R}$, we will use the notations $B(\dot{H}^s;I)$ and $B'(\dot{H}^{-s};I)$. 

With these notations we recall the following Strichartz-type estimates.

\begin{lemma}\label{Lemma-Str}
Let $I\subset\mathbb{R}$ be an interval and $t_0\in I$.
	The following statements hold.
 \begin{itemize}
\item [(i)] (\textbf{Linear estimates}).
\begin{equation}\label{SE1}
\| e^{it\Delta^2}f \|_{B(L^2;I)} \lesssim\|f\|_{L^2},
\end{equation}
\begin{equation}\label{SE2}
\| e^{it\Delta^2}f \|_{B(\dot{H}^s;I)} \lesssim \|f\|_{\dot{H}^s}.
\end{equation}
\item[(ii)] (\textbf{Inhomogeneous estimates}).
\begin{equation}\label{SE3}					 
\left \| \int_{t_0}^t e^{i(t-t')\Delta^2}g(\cdot,t') dt' \right \|_{B(L^2;I) } \lesssim\|g\|_{B'(L^2;I)},
\end{equation}
\begin{equation}\label{SE5}
\left \| \int_{t_0}^t e^{i(t-t')\Delta^2}g(\cdot,t') dt' \right \|_{B(\dot{H}^s;I) } \lesssim\|g\|_{B'(\dot{H}^{-s};I)}.
\end{equation}
\end{itemize}
\end{lemma} 
\begin{proof}
See for instance \cite{Pausader07} and \cite{Guo}.
\end{proof}

We also recall the following result.

\begin{proposition}\label{estimativanaolinear} Consider $N\geq 3$. Let $I\subset\mathbb{R}$ be an interval and $t_0\in I$. Suppose that\\ $F\in L^1_{loc}(I,H^{-4})$ and let  $u\in C(I,H^{-4})$ be a solution of
$$
u(t)= e^{i(t-t_0)\Delta^2}u(t_0)+i\lambda \int_{t_0}^t e^{i(t-t')\Delta^2}F(\cdot,t')dt'.
$$
 Then, for any $B$-admissible pair $(q,r)$ and $s \geq 0$, we obtain
\begin{equation}\label{ESB2}
\left\|D^s u\right\|_{L^{q}_{I}L_x^{r}} \lesssim \left\| D^{s}u(t_0)\right\|_{L^{2}}+\left\|D^{s-1} F\right\|_{L^{2}_{I}L_x^{\frac{2N}{N+2}}}.
\end{equation}
In particular, when $s=2$,
\begin{equation}\label{EstimativaImportante}
\left\|\Delta u\right\|_{L^{q}_{I}L_x^{r}} \lesssim \left\| \Delta u(t_0)\right\|_{L^{2}}+\left\|\nabla F\right\|_{L^{2}_{I}L_x^{\frac{2N}{N+2}}}.
\end{equation}
\end{proposition}
\begin{proof}
See Proposition 2.3 in \cite{GUZPAS}.
\end{proof}


Throughout the paper, we set $B=\{ x\in \mathbb{R}^N;|x|\leq 1\}$ and $B^C=\mathbb{R}^N\backslash B$. Recall that
$$\||x|^{-b}\|_{L^\gamma(B)}<+\infty\;\;\;\textnormal{if}\;\;\frac{N}{\gamma}-b>0\quad \textnormal{and}\quad \||x|^{-b}\|_{L^\gamma(B^C)}<+\infty\;\;\;\textnormal{if}\;\;\frac{N}{\gamma}-b<0.
$$
In order to get our nonlinear estimates, we need some pointwise estimate on 
 $F(x,z)=|x|^{-b}|z|^\alpha z$. The following will be enough to our purposes (see details in \cite[Remark 2.6]{CARLOS} and \cite[Remark 2.5]{paper2})
\begin{equation}\label{FEI}
 |F(x,z)-F(x,w)|\lesssim |x|^{-b}\left( |z|^\alpha+ |w|^\alpha \right)|z-w|
\end{equation}
and
\begin{equation}\label{SECONDEI}
\left|\nabla \left(F(x,z)-F(x,w)\right)\right|\lesssim  |x|^{-b-1}(|z|^{\alpha}+|w|^{\alpha})|z-w|+|x|^{-b}|z|^\alpha|\nabla (z- w)|+E, 
\end{equation}
where 
\begin{eqnarray*}
 E &\lesssim& \left\{\begin{array}{cl}
 |x|^{-b}\left(|z|^{\alpha-1}+|w|^{\alpha-1}\right)|\nabla w||z-w|, & \textnormal{if}\;\;\;\alpha > 1 \vspace{0.2cm} \\ 
|x|^{-b}|\nabla w||z-w|^{\alpha}, & \textnormal{if}\;\;\;0<\alpha\leq 1.
\end{array}\right.
\end{eqnarray*}

\ 

\section{\bf Global Well-Posedness and Stability for the intercritical case}\label{Sec3}
	
\ The goal of this section is to study the global well-posedness of the Cauchy problem \eqref{IBNLS}. We first turn attention to the proof of Theorem \ref{GWPH2}. The main point is to establish suitable estimates on the term  $F(x,u,v)=|x|^{-b}|u|^\alpha v$.  To simplify notation, when $u=v$, we denote $F(x,u,v)$ by $F(x,u)$.

\begin{lemma}\label{lemmaglobal1} 
Let $N\geq 3$ and $0<b<\min\{\frac{N}{2},4\}$. If $ \frac{8-2b}{N}<\alpha<4^*$ then the following statements hold:
\begin{itemize}
\item [(i)] $\left \|\chi_B F(x,u,v) \right\|_{B'(\dot{H}^{-s_c})}+
\left \|\chi_{B^C}F(x,u,v) \right\|_{B'(\dot{H}^{-s_c})} \lesssim \| u\|^{\theta}_{L^\infty_tH^2_x}\|u\|^{\alpha-\theta}_{B(\dot{H}^{s_c})} \|v\|_{B(\dot{H}^{s_c})};
$
\item [(ii)] $\left\|\chi_B F(x,u,v) \right\|_{B'(L^2)} + \left\|\chi_{B^C} F(x,u,v) \right\|_{B'(L^2)} 
\lesssim \| u\|^{\theta}_{L^\infty_tH^2_x}\|u\|^{\alpha-\theta}_{B(\dot{H}^{s_c})} \| v\|_{B(L^2)}$,
\end{itemize} 
where  $\theta\in (0,\alpha)$ is a sufficiently small number.	
\begin{proof} See Lemma 4.2 in \cite{GUZPAS}. 
\end{proof}
\end{lemma}


\begin{lemma}\label{improvedlemma} 
Assume $N=5,6,7$, $0<b<\min\{\frac{N}{2},4\}$ and $ \frac{8-2b}{N}<\alpha<4^*$ $(\alpha<7-2b$ if $N=5$). Then,
\begin{equation}\label{LG1} 
\left\|\nabla F(x,u) \right\|_{L_t^2L_x^{\frac{2N}{N+2}}}\lesssim \| u\|^{\theta}_{L^\infty_tH^2_x}\|u\|^{\alpha-\theta}_{B(\dot{H}^{s_c})} \| \Delta u\|_{B(L^2)},
\end{equation}
where $\theta\in (0,\alpha)$ is a sufficiently small number.
\begin{proof}
First we  consider $N=6,7$. Define the following numbers:
\begin{align}\label{pairN>=5a}
\bar{a}\;=\;\frac{8\alpha(\alpha+1-\theta)}{8-2b-\alpha(N-4)}\;,\;\;\;\;\bar{r}=\frac{2\alpha N(\alpha+1-\theta)}{\alpha(N+4-2b)-2\theta (4-b)}
\end{align}
and
\begin{align}\label{pairN>=5b}
\bar{q}\;=\;\frac{8\alpha(\alpha+1-\theta)}{\alpha(N\alpha-4+2b)-\theta(N\alpha -8+2b)},
\end{align}
$\theta>0$ is sufficiently small. It follows easily that $(\bar{q},\bar{r})$ is $B$-admissible and $(\bar{a},\bar{r})$ is $\dot{H}^{s_c}$-biharmomic admissible. In addition,
\begin{equation}\label{holglo}
\frac{1}{2}=\frac{\alpha-\theta}{\bar{a}}+\frac{1}{\bar{q}}.
\end{equation}
Note that
$$
|\nabla F(x,u)|\lesssim |x|^{-b}\left(|u|^\alpha |\nabla u|+|x|^{-1}|u|^\alpha |u|\right).
$$
Thus, an application of the Hardy-Littlewood inequality (Lemma \ref{Hardy} with $\rho=1$ and $p=q$) gives $\left\||x|^{-1}(|u|^\alpha u)\right\|_{L^\beta} \lesssim \left\||u|^\alpha \nabla u\right\|_{L^\beta}$, where $1<\beta<N$. Let $A$ denotes either $B$ or $B^C$. The H\"older inequality and the Sobolev embedding lead to  
\[
\begin{split}
\left\|\nabla F(x,u)\right\|_{L_x^{\frac{2N}{N+2}}(A)}&\lesssim  \||x|^{-b}\|_{L^\gamma(A)}\||u|^\alpha \nabla u\|_{L^\beta} \lesssim \||x|^{-b}\|_{L^\gamma(A)}  \|u\|^\theta_{L_x^{\theta r_1}}  \|u\|^{\alpha-\theta}_{L_x^{\bar{r}}} \| \nabla u \|_{L_x^{r_2}}\\
&\lesssim  \||x|^{-b}\|_{L^\gamma(A)} \|u\|^\theta_{L_x^{\theta r_1}}  \|u\|^{\alpha-\theta}_{L_x^{\bar{r}}} \| \Delta u \|_{L_x^{\bar{r}}},
\end{split}
\]
provided
\begin{equation}\label{LGR1} 
\frac{N}{\gamma}=\frac{N+2}{2}-\frac{N}{\beta}\;, \qquad \;\frac{N}{\beta}=\frac{N}{r_1}+\frac{N(\alpha-\theta)}{\bar{r}}+\frac{N}{r_2}\;,\qquad \; 1=\frac{N}{\bar{r}}-\frac{N}{r_2}, \qquad \bar{r}<N.
\end{equation}
Using the definition of the number $\bar{r}$ one has
\begin{equation}\label{LGr2}
\frac{N}{\gamma}-b=\frac{\theta(4-b)}{\alpha}-\frac{N}{r_1}.
\end{equation}
Consequently, if $A = B$ we choose $r_1$ such that $\theta r_1=\frac{2N}{N-4}$, which gives $\frac{N}{\gamma}-b=\theta (2-s_c) > 0$ (recall that
$s_c < 2$). On the other hand, if $A = B^C$ we choose $r_1$ satisfying $\theta r_1 =2$, yielding $\frac{N}{\gamma}-b=-\theta s_c< 0$. Thus, in both cases the quantity $\||x|^{-b}\|_{L^\gamma(A)}$ is finite and, by Sobolev embedding, $H^2 \hookrightarrow
L^{\theta r_1}$. Therefore, in view of \eqref{holglo} we deduce that
$$
\left\|\nabla F(x,u)\right\|_{L^2L_x^{\frac{2N}{N+2}}}\lesssim \|u\|^\theta_{L^\infty_t H^2_x} \|u\|^{\alpha-\theta}_{L^{\bar{a}}_tL^{\bar{r}}_x}\|\Delta u\|_{L^{\bar{q}}_tL^{\bar{r}}_x}.
$$
To complete the proof we need to verify that $1<\beta <N$ and $\bar{r}<N$. Indeed, by using the first relation in \eqref{LGR1} and the values of $r_1$ above, we obtain
\begin{equation*}
\beta\;=\;
\begin{cases}
\frac{2N}{N+2-2b-2\theta (2-s_c)} \;\; \textnormal{if}\;\;A=B\\
\frac{2N}{N+2-2b+2\theta s_c} \;\quad \quad \textnormal{if}\;\;A=B^C.
\end{cases}    
\end{equation*}
It is easy to see that $1<\beta<N$ provided $b<\frac{N}{2}$ and $\theta$ is sufficiently small. Furthermore, note that $\bar{r}<N$ is equivalent to $\alpha <\frac{N+2-2b}{2}$, which is true if $N=6,7$, in view of our hypothesis $\alpha<\frac{8-2b}{N-4}$. 

 In the sequel, we consider the case $N=5$. Define,
for $\varepsilon>0$ sufficiently small, 
\begin{align*}
q_\varepsilon\;=\;\frac{8}{3-2\varepsilon}\;,\;\;\;\;r_\varepsilon=\frac{5}{1+\varepsilon},
\end{align*}
and
\begin{align*}
a\;=\;\frac{8(\alpha-\theta)}{1+2\varepsilon}\;,\;\;\;\;\;\;\;r=\frac{10\alpha(\alpha-\theta)}{\alpha(7-2b)-2\theta (4-b)-2\varepsilon \alpha}.
\end{align*}
Note that $r_\varepsilon<5$ and since $\alpha<7-2b$ we have  $r<10$, for $\varepsilon$ and $\theta$ sufficiently small. Moreover, an easy computation shows that $(a,r)$ is $\dot{H}^{s_c}$-biharmonic admissible and $(q_\varepsilon,r_\varepsilon)$ is $B$-admissible. Similarly as before we have
\[
\begin{split}
\left\|\nabla F(x,u)\right\|_{L_x^{\frac{2N}{N+2}}(A)} &\lesssim \||x|^{-b}\|_{L^\gamma(A)}\|u\|^\theta_{L_x^{\theta r_1}}  \|u\|^{\alpha-\theta}_{L_x^{r}}   \| \nabla u \|_{L_x^{r_3}}    \\
&\lesssim   \|u\|^\theta_{L_x^{\theta r_1}}  \|u\|^{\alpha-\theta}_{L_x^{r}} \| \Delta u \|_{L_x^{r_\varepsilon}} ,
\end{split}
\]
where 
\begin{equation*} 
\frac{N}{\gamma}=\frac{N}{2}+1-\frac{N}{r_1}-\frac{N(\alpha-\theta)}{r}-\frac{N}{r_3}=\frac{N}{2}+1-\frac{N}{r_1}-\frac{N(\alpha-\theta)}{r}-(\frac{N}{r_\varepsilon}-1).
\end{equation*}
Thus,  using the values of $r$ and $r_\varepsilon$, we deduce
$$
\frac{N}{\gamma}-b=\frac{\theta(4-b)}{\alpha}-\frac{N}{r_1},
$$
which is the same relation as in \eqref{LGr2}. Since $\frac{1}{2}=\frac{\alpha-\theta}{a}+\frac{1}{q_\varepsilon}$, the rest of the proof runs as in the previous case.
\end{proof}	 
\end{lemma}

As we already pointed out, in dimension $N=5$, we have the extra condition $\alpha<7-2b$. However, if we are interested in the global well-posedness for $\alpha$ in the intercritical range, that is, $\frac{8-2b}{5}<\alpha<8-2b$ we should take $b$ smaller than in the previous lemma. Here, we used the numbers defined in \eqref{pairN=5a}, i.e.,  
\begin{align*}
a^*\;=\;\frac{8\alpha(\alpha+1-\theta)}{8-2b-\alpha(N-4)+2\varepsilon\alpha}\;,\;\;\;\;r^*=\frac{2\alpha N(\alpha+1-\theta)}{\alpha(N+4-2b)-2\theta (4-b)-2\varepsilon \alpha}.
\end{align*}
\begin{lemma}\label{LG3} 
Assume $N=5$, $0<b\leq \frac{3}{2}$ and $ \frac{8-2b}{N}<\alpha<\frac{8-2b}{N-4}$. Then,
\begin{itemize}
\item [(i)] $\left\|F(x,u,v) \right\|_{B'(L^2)}\lesssim \| u\|^{\theta}_{L^\infty_tH^2_x}\|u\|^{\alpha-\theta}_{L^{a^*}_tL_x^{r^*}} \| v\|_{B(L^2)}$
\item [(ii)] $\left\|\Delta F(x,u) \right\|_{B'(L^2)}\lesssim \| u\|^{\theta}_{L^\infty_tH^2_x}\|u\|^{\alpha-\theta}_{L^{a^*}_tL_x^{r^*}} \| \Delta  u\|_{B(L^2)}$
\item [(iii)] $\left\|D^{s_c} F(x,u) \right\|_{B'(L^2)}\lesssim \| u\|^{\theta}_{L^\infty_tH^2_x}\|u\|^{\alpha-\theta}_{L^{a^*}_tL_x^{r^*}} \| u\|^{1-\eta}_{B(L^2)} \| \Delta  u\|^{\eta}_{B(L^2)}, \quad \eta=\frac{s_c}{2}$, 
\end{itemize}
where $\theta\in (0,\alpha)$ is a sufficiently small number.
\begin{proof}
We first note that ($a^*,r^*)$ satisfies the relation $\frac{4}{a^*}=\frac{N}{2}-\frac{N}{r^*}-s_c$ but it is not an $\dot{H}^{s_c}$ - biharmonic admissible pair. To work with admissible pairs, let us start by defining the following numbers:

\begin{align}\label{pairN=5b}
q^*\;=\;\frac{8\alpha(\alpha+1-\theta)}{\alpha(N\alpha-4+2b)-\theta(N\alpha -8+2b)+2\varepsilon\alpha},\;\;\quad q_\varepsilon=\frac{4}{2-\varepsilon},\quad\; r_\varepsilon=\frac{2N}{N-4+2\varepsilon},
\end{align}
were $\varepsilon,\theta>0$ are sufficiently small. It is east to check that ($q^*,r^*)$ and ($q_\varepsilon,r_\varepsilon)$  are $B$-admissible. Moreover,
\begin{equation}\label{N=5}
\frac{1}{q_\varepsilon'}=\frac{\alpha-\theta}{a^*}+\frac{1}{q^*}.
\end{equation}

 First we prove  (ii). Observe that 
$$
|\Delta F(x,u)|\lesssim |x|^{-b}\Big[|\Delta(|u|^\alpha u)|+|x|^{-2}||u|^\alpha u|+|x|^{-1}|\nabla (|u|^\alpha u)|\Big]
$$
 and Lemma \ref{Hardy} implies  $\left\||x|^{-2}(|u|^\alpha u)\right\|_{L_x^\beta} \lesssim \left\|\Delta (|u|^\alpha u)\right\|_{L_x^\beta}$ and $\left\||x|^{-1}\nabla(|u|^\alpha u)\right\|_{L_x^\beta} \lesssim \left\|\Delta (|u|^\alpha u)\right\|_{L_x^\beta}$, for any  $1<\beta<\frac{N}{2}$. Let $A$ denotes either $B$ or $B^C$. Applying the H\"older inequality we get
\begin{equation}\label{a1}
    \begin{split}
\left\|\Delta F(x,u)\right\|_{L_x^{r'_\varepsilon}}\lesssim  \||x|^{-b}\|_{L_x^\gamma(A)}\|\Delta \left(|u|^\alpha  u\right)\|_{L_x^\beta}
\end{split}
\end{equation}
where
\begin{equation}\label{a2}
\frac{N}{\gamma}=\frac{N+4-2\varepsilon}{2}-\frac{N}{\beta}.
\end{equation}
Next we observe that, for $\alpha>1$,
\[
|\Delta \left(|u|^\alpha  u\right)|\lesssim ||u|^\alpha\Delta u|+|u|^{\alpha-1}|\nabla u|^2.
\]
Thus, an application of Lemma \ref{GNinequality} and interpolation give
\begin{equation*}
\begin{split}
  \|\Delta \left(|u|^\alpha  u\right)\|_{L_x^\beta}&\lesssim \|u\|^{\alpha}_{L_x^{p_1}}\|\Delta u\|_{L_x^{p_2}} +\|u\|^{\alpha-1}_{L_x^{p_1}}\|\nabla u\|_{L_x^{p_3}}^2\\
  &\lesssim \|u\|^{\alpha}_{L_x^{p_1}}\|\Delta u\|_{L_x^{p_2}} +\|u\|^{\alpha-1}_{L_x^{p_1}}\|u\|_{L^{p_1}}\|\Delta u\|_{L_x^{p_2}}\\
  &=\|u\|^{\alpha}_{L_x^{p_1}}\|\Delta u\|_{L_x^{p_2}}\\
  &\lesssim \|u\|_{L_x^{r_3}}^{\kappa \alpha}\|u\|_{L_x^{r_2}}^{(1-\kappa) \alpha}\|\Delta u\|_{L_x^{p_2}},
\end{split}
\end{equation*}
where, for $\kappa\in(0,1)$,
\[
\frac{1}{\beta}=\frac{\alpha}{p_1}+\frac{1}{p_2}=\frac{\alpha-1}{p_1}+\frac{2}{p_3}=\frac{\eta\alpha}{r_3}+\frac{(1-\eta)\alpha}{r_2}+\frac{1}{p_2}.
\]
By choosing $\kappa=\theta/\alpha$, $r_2=p_2=r^*$, and $r_3=\theta r_1$ we finally deduce that
\begin{equation}\label{a3}
     \|\Delta \left(|u|^\alpha  u\right)\|_{L_x^\beta}\lesssim  \|u\|^\theta_{L_x^{\theta r_1}}  \|u\|^{\alpha-\theta}_{L_x^{r^*}} \| \Delta u \|_{L_x^{r^*}}
\end{equation}
with
\begin{equation}\label{a4}
    \frac{1}{\beta}=\frac{1}{r_1}+\frac{\alpha-\theta}{r^*}+\frac{1}{r^*}.
\end{equation}
Thus, from \eqref{a1}-\eqref{a3}, we get

\[
\left\|\Delta F(x,u)\right\|_{L_x^{r'_\varepsilon}}\lesssim \||x|^{-b}\|_{L_x^\gamma(A)} \|u\|^\theta_{L_x^{\theta r_1}}  \|u\|^{\alpha-\theta}_{L_x^{r^*}} \| \Delta u \|_{L_x^{r^*}}
\]
where, combining \eqref{a2} and \eqref{a4} with the value of $r^*$,
\begin{equation*}
\frac{N}{\gamma}-b=\frac{\theta(4-b)}{\alpha}-\frac{N}{r_1}.
\end{equation*}
which is the same relation as in \eqref{LGr2}. Arguing as in the proof of Lemma \ref{improvedlemma} and using \eqref{N=5}, it follows that
$$
\left\|\Delta F(x,u)\right\|_{L_t^{q'_\varepsilon}L_x^{r'_\varepsilon}}\lesssim \|u\|^\theta_{L^\infty H^2_x}  \|u\|^{\alpha-\theta}_{L^{a^*}_tL_x^{r^*}} \| \Delta u \|_{L^{q^*}_tL_x^{r^*}}\lesssim \|u\|^\theta_{L^\infty H^2_x}  \|u\|^{\alpha-\theta}_{L^{a^*}_tL_x^{r^*}} \| \Delta u \|_{B(L^2)}.
$$
To finish the proof of (ii), we check that $1<\beta<\frac{N}{2}$. Similarly as as in the proof of Lemma \ref{improvedlemma}, we have
\begin{equation*}
\beta\;=\;
\begin{cases}
\frac{2N}{N+4-2b-2\theta (2-s_c)-2\varepsilon}, \;\; \textnormal{if}\;\;A=B,\\
\frac{2N}{N+4-2b+2\theta s_c-2\varepsilon}, \;\quad \quad \textnormal{if}\;\;A=B^C,
\end{cases}    
\end{equation*}
which immediately implies the desired provided $\varepsilon$ and $\theta$ are sufficiently small.

 The proof of part (i) is close to that of (ii) with the advantage that it suffices to apply H\"older's inequality. So we omit additional details.
 We conclude with (iii). Applying Lemma \ref{GNinequality}
 we get
\begin{eqnarray*}
\left\|D^{s_c} \left(|x|^{-b}|u|^\alpha u\right) \right\|_{L_x^{r_\varepsilon'}}\lesssim  \left\| \left(|x|^{-b}|u|^\alpha u\right) \right\|^{1-\eta}_{L_x^{r_\varepsilon'}} \left\| \Delta\left(|x|^{-b}|u|^\alpha u\right) \right\|^{\eta}_{L_x^{r_\varepsilon'}},  \qquad \eta=\frac{s_c}{2}.
\end{eqnarray*}
Thus, H\"older's inequality and parts (i) and (ii) yield
\[
\begin{split}
\left\|D^{s_c} \left(|x|^{-b}|u|^\alpha u\right) \right\|_{L_t^{q_\varepsilon'} L_x^{r_\varepsilon'}}&\lesssim \| u\|^{\theta}_{L^\infty_tH^2_x}\|u\|^{\alpha-\theta}_{L^{a^*}_tL_x^{r^*}} \| u\|^{1-\eta}_{B(L^2)} \| \Delta  u\|^{\eta}_{B(L^2)},
\end{split}
\]
completing the proof of the lemma.
\end{proof}
\end{lemma}

Now, with all the previous lemmas in hand we are in a position to prove Theorems \ref{GWPH2} and \ref{GWPH2'}.  Once estimate \eqref{LG1} has been established, the proof of Theorem \ref{GWPH2} is the the same as in \cite[Theorem $1.6$]{GUZPAS}; so we omit the details. We only consider the case $N=5$ (Theorem \ref{GWPH2'}). 

\begin{proof}[\bf{Proof of Theorem \ref{GWPH2'}}] 
 Let $S$ be the set of all functions $u:\mathbb{R}\times \mathbb{R}^5 \to\mathbb{C}$ such that
$$
\|u\|_{L^{a^*}_tL_x^{r^*}}\leq 2\|e^{it\Delta^2}u_0\|_{L^{a^*}_tL_x^{r^*}}\;\quad \textnormal{and}\; \quad \|\langle \Delta \rangle u\|_{B(L^2)} \leq 2c\|u_0\|_{H^2},
$$ 
where $\|\langle \Delta \rangle u\|_{B(L^2)}:=\|u\|_{B(L^2)}+\|\Delta u\|_{B(L^2)}$.
The idea is to use the contraction mapping principle to show that the map $G$ in \eqref{OPERATOR} is a contraction on $S$ equipped with the metric 
$$
d(u,v)=\|u-v\|_{B(L^2)}.
$$
Define $p^*=\frac{10 (\alpha+1-\theta)}{5\alpha+1+2s_c-5\theta -2\varepsilon}$. It is easy to see that ($a^*, p^*)$ is $B$-admissible and 
$s_c=\frac{5}{p^*}-\frac{5}{r^*}$. In addition, $p^*<\frac{5}{s_c}$ and $ 2 <p^* <10$. Hence, by using the Sobolev embedding, it follows that
\begin{eqnarray*}
\|G(u)\|_{L_t^{a^*}L_x^{r^*}}&\leq&  \|e^{it\Delta^2}u_0\|_{L_t^{a^*}L_x^{r^*}}+\left\|D^{s_c}\int_0^t e^{i(t-s)\Delta^2} F(x,u)ds\right\|_{L_t^{a^*}L_x^{p^*}}\\
&\leq&  \|e^{it\Delta^2}u_0\|_{L_t^{a^*}L_x^{r^*}}+\left\|D^{s_c} F(x,u)\right\|_{B'(L^2)},
\end{eqnarray*}
where we have used the Strichartz estimate \eqref{SE3}, since ($a^*,p^*$) is $B$-admissible. An application of Lemma \ref{LG3} (iii) yields, for any $u\in S$,
\begin{equation}\label{TGHS1}
\begin{split}
\|G(u)\|_{L^{a^*}_tL_x^{r^*}}
&\leq \|e^{it\Delta^2}u_0\|_{L^{a^*}_tL_x^{r^*}} +c\| u \|^\theta_{L^\infty_tH^2_x}\| u \|^{\alpha-\theta}_{L^{a^*}_tL_x^{r^*}}\|\langle \Delta \rangle u\|_{B(L^2)} \\ 
&\leq  \|e^{it\Delta^2}u_0\|_{L^{a^*}_tL_x^{r^*}}+2^{\alpha+1}c^{\theta+2}\eta^{\theta+1}\| e^{it\Delta^2}u_0 \|^{\alpha-\theta}_{L^{a^*}_tL_x^{r^*}},
\end{split}
\end{equation}
where we have used the fact that $(\infty,2)$ is $B$-admissible to see that $\|u\|_{L^\infty_tH^2_x}\leq \|\langle\Delta\rangle u\|_{B(L^2)}$.

On the other hand, estimates \eqref{SE1} and \eqref{SE3} imply
\begin{equation*}\label{GHs21}
\|G(u)\|_{B(L^2)}\leq c\|u_0\|_{L^2}+ c\| F(x,u) \|_{B'(L^2)}
\end{equation*}
and	
\begin{equation*}
\|\Delta G(u)\|_{B(L^2)}\leq c \|\Delta u_0\|_{L^2}+ c\|\Delta F(x,u)\|_{B'(L^2)}.
\end{equation*}
Therefore, from Lemma \ref{LG3},
\begin{equation}\label{TGHS11}
\begin{split}
\|G(u)\|_{B(L^2)}+\|\Delta G(u)\|_{B(L^2)}&\leq  c\|u_0\|_{H^2}+c\| u \|^\theta_{L^\infty_tH^2_x}\| u \|^{\alpha-\theta}_{L^{a^*}_tL_x^{r^*}}(\|\Delta u\|_{B(L^2)}+\|u\|_{B(L^2)})\\
& \leq  c\|u_0\|_{H^2}+c2^{\alpha+1}c^{\theta+1}\|u_0\|_{H^2_x}^{\theta+1} \| e^{it\Delta^2}u_0 \|^{\alpha-\theta}_{L^{a^*}_tL_x^{r^*}}\\
& \leq  c\|u_0\|_{H^2}+c2^{\alpha+1}c^{\theta+1}\eta^\theta \| e^{it\Delta^2}u_0 \|^{\alpha-\theta}_{L^{a^*}_tL_x^{r^*}}\|u_0\|_{H^2_x}
\end{split}
\end{equation}
 Now if $\| e^{it\Delta^2}u_0 \|_{L^{a^*}_tL_x^{r^*}}<\delta$ with
\begin{equation}\label{WD1}
\delta\leq \min\left\{\sqrt[\alpha-1-\theta]{\frac{1}{2c^{\theta+2}2^{\alpha+1}\eta^{\theta+1}}}     , \sqrt[\alpha-\theta]{ \frac{1}{2c^{\theta+1}2^{\alpha+1}\eta^\theta}}\right\},
\end{equation}
 it follows from \eqref{TGHS1} and \eqref{TGHS11} that 
$$\|G(u)\|_{L^{a^*}_tL_x^{r^*}}\leq 2\| e^{it\Delta^2}u_0 \|_{L^{a^*}_tL_x^{r^*}}\quad \mbox{and}\quad \|\langle \Delta \rangle G(u)\|_{B(L^2)}\leq 2c\|u_0\|_{H^2},$$
which means to say   $G(u)\in S$.

\ To show that $G$ is a contraction on $S$, we repeat the above computations taking into account \eqref{FEI}. Indeed, for any $u,v\in S$ we may show that
\begin{equation*}\label{C2GH1}
\|G(u)-G(v)\|_{B(L^2)}\leq 2^{\alpha+1}c^{\theta+1} \| u_0 \|^\theta_{H^2}\|e^{it\Delta^2}u_0 \|^{\alpha-\theta}_{L^{a^*}_tL_x^{r^*}} \|u-v\|_{B(L^2)}.
\end{equation*}
From the last inequality and \eqref{WD1} it follows that
$$
d(G(u),G(v)) \leq 2^{\alpha+1}c^{\theta+1} \| u_0 \|^\theta_{H^2}\|e^{it\Delta^2}u_0 \|^{\alpha-\theta}_{L^{a^*}_tL_x^{r^*}}\;d(u,v)\leq \frac{1}{2}\;d(u,v),
$$
which means that  $G$ is also a contraction. Therefore, an application of the contraction mapping principle, gives that $G$ has a unique fixed point $u\in S$, which in turn is a global solution of \eqref{IBNLS}. The proof of the theorem is thus completed.
\end{proof}

\ The proof of Propositions \ref{SCATTERSH1} and \ref{LTP} is similar to those in \cite[Theorem 1,7]{GUZPAS} and \cite[Theorem 1.9]{GUZPAS}, respectively; so we omit the details.

\ 

\section{\bf Well-Posedness and Stability for the energy-critical case}

\ In this section study the Cauchy problem \eqref{IBNLS}, for $\alpha=\frac{8-2b}{N-4}$, i.e., energy critical case. We first consider the local well-possedness result (Theorem \ref{GWPCH2}) and in the sequel we study a stability result (Theorem \ref{LTPC}). The theorems follow from a contraction mapping argument based on the Strichartz estimates. Before showing the results we establish a technical lemma.
\begin{lemma}\label{LemCritico}
Suppose that $\alpha$ and $b$ are as in Theorem \ref{GWPCH2}. Then the following statements hold.
\begin{itemize}
\item [(i)] $\left\||x|^{-b}|f|^\alpha \nabla g\right\|_{L^2_IL^{\frac{2N}{N+2}}}\;\lesssim \;\|\Delta f\|^{b}_{B(L^2;I)} \|f\|^{\alpha-b}_{B(I)}\|\Delta g\|_{B(L^2;I)}$;
\item [(ii)] $\left\||x|^{-b}|f|^{\alpha -1}h\nabla g\right\|_{L^2_IL^{\frac{2N}{N+2}}}\;\lesssim \;\|\Delta f\|^{b}_{B(L^2;I)} \|f\|^{\alpha-1-b}_{B(I)}\|h\|_{B(I)}\|\Delta g\|_{B(L^2;I)}$,\;\;\;if\;\;\;$\alpha>1$;
\item[(iii)]   $\left\||x|^{-b}|f|^\alpha (|x|^{-1}g)\right\|_{L^2_IL^{\frac{2N}{N+2}}}\;\lesssim \;\|\Delta f\|^{b}_{B(L^2;I)} \|f\|^{\alpha-b}_{B(I)}\|\Delta g\|_{B(L^2;I)}$,
\end{itemize}
with $B(I)$ defined in \eqref{BI}.
\begin{proof}
First note that the condition $b<\frac{12-N}{N-2}$ implies that $\alpha-b-1>0$. Next we define the following numbers
\begin{equation*}
q=\frac{2(N+4)(b+1)}{b(N-2)+N-4}\qquad r=\frac{2N(N+4)(b+1)}{N^2+b(N^2+8)+16}\quad \textnormal{and}\quad \bar{r}=\frac{2(N+4)}{N-4}.
\end{equation*}
It is easy to see that ($q,r)$ is $B$-admissible. 

Let us prove (i). By setting
$$
\beta=\frac{Nr}{N-r},
$$
we may check that
\begin{equation*}
\beta<N, \qquad\frac{N+2}{2N}=\frac{\alpha-b}{\bar{r}}+\frac{b}{\beta}+\frac{1}{\beta},\qquad  \frac{1}{2}=\frac{\alpha-b}{\bar{r}}+\frac{b}{q}+\frac{1}{q}, \qquad \textnormal{and}\qquad 1=\frac{N}{r}-\frac{N}{\beta}.  
\end{equation*}
Now, applying H\"older's inequality and Lemma \ref{Hardy} (with $\rho=1$) we deduce
\[
\begin{split}
\left\||x|^{-b}|f|^\alpha \nabla g\right\|_{L^2_IL^{\frac{2N}{N+2}}}
&\lesssim  \||x|^{-1}f\|_{L_I^qL_x^\beta}^{b} \|f\|^{\alpha-b}_{L_I^{\bar{r}}L_x^{\bar{r}}}\|\nabla g\|_{L^q_IL^\beta_x}\\
&\lesssim  \|\nabla f\|_{L_I^qL_x^\beta}^{b} \|f\|^{\alpha-b}_{L_I^{\bar{r}}L_x^{\bar{r}}}\|\nabla g\|_{L^q_IL^\beta_x}\\
&\lesssim \|\Delta f\|^{b}_{L^q_IL^r_x} \|f\|^{\alpha-b}_{B(I)}\|\Delta g\|_{L^q_IL^r_x},
\end{split}
\]
where in the last inequality we used the Sobolev embedding. This completes the proof of (i) since ($q,r)$ is $B$-admissible. 

The proof of (ii) follows the same lines as above just replacing $|f|^{\alpha}$ by $|f|^{\alpha-1}h$. Estimate (iii) also follows as in (i) taking into account that, from Lemma \ref{Hardy}, $\||x|^{-1}g\|_{L_x^\beta}\lesssim \|\nabla g\|_{L_x^\beta}$. The proof of the lemma is thus completed.
\end{proof}
\end{lemma}

\begin{proof}[\bf{Proof of Theorem \ref{GWPCH2}}] 
As before we will use the contraction mapping principle to show that operator  $G$ defined in \eqref{OPERATOR} is a contraction on a suitable metric space. For $K$ and $\rho$ positive constans to be determined later, let $S_{\rho,K}$ be the set of all functions $u:I\times\mathbb{R}^5\ \to\mathbb{C}$ such that
$$
 \|u\|_{B(I)}\leq \rho \quad \textnormal{and} \quad \|\Delta u\|_{B(L^2;I)}\leq K.
$$
The space  $S_{\rho,K}$ will be equipped with the metric 
$$
d(u,v):=\|\Delta (u-v)\|_{B(L^2;I)}+\| u-v\|_{B(I)}.
$$
 
Applying the Sobolev embedding and estimate \eqref{EstimativaImportante}, it follows that
\begin{equation}\label{trii}
\begin{split}
\|G(u)\|_{B(I)}&\leq  \|e^{it\Delta^2}u_0\|_{B(I)}+c\left\|\Delta \int_0^t e^{i(t-s)\Delta^2} F(x,u)ds\right\|_{L_I^{\frac{2(N+4)}{N-4}}L_x^{\frac{2N(N+4)}{N^2+16}}}\\
&\leq  \|e^{it\Delta^2}u_0\|_{B(I)}+c\left\|\nabla F(x,u)\right\|_{L_I^{2}L_x^{\frac{2N}{N+2}}}.
\end{split}
\end{equation}
In addition, \eqref{EstimativaImportante} also gives
\begin{equation*}
\|\Delta G(u)\|_{B(L^2;I)}\lesssim \|\Delta u_0\|_{L^2}+ \|\nabla F(x,u)\|_{L^2_IL_x^{\frac{2N}{N+2}}}.
\end{equation*}
Now, by recalling from \eqref{SECONDEI} that
\begin{equation*}
|\nabla F (x,u)| \leq  |x|^{-b}|u|^{\alpha} |\nabla u| + |x|^{-b}|u|^{\alpha} ||x|^{-1} u|
\end{equation*}
and an application of Lemma \ref{LemCritico} leads to
\begin{eqnarray*}
\|\nabla F(x,u)\|_{L^2_IL_x^{\frac{2N}{N+2}}} &\leq & \|u\|^{\alpha-b}_{B(I)}\|\Delta u\|^{b+1}_{B(L^2;I)}.
\end{eqnarray*}
Thus,
\begin{equation*}
\begin{split}
\|G(u)\|_{B(I)}
&\leq \|e^{it\Delta^2}u_0\|_{B(I)} +c \|u\|^{\alpha-b}_{B(I)}\|\Delta u\|^{b+1}_{B(L^2;I)}\\
&\leq  \delta +c \rho^{\alpha-b}K^{b+1} 
\end{split}
\end{equation*}
and 
\begin{equation*}
\begin{split}
\|\Delta G(u)\|_{B(L^2;I)}&\leq  c\|u_0\|_{\dot{H}^2}+c\|u\|^{\alpha-b}_{B(I)}\|\Delta u\|^{b+1}_{B(L^2;I)}\\
&\leq  cA +c \rho^{\alpha-b}K^{b+1}.
\end{split}
\end{equation*}
By choosing $K= 2Ac$, $\rho$  small enough  so that $\max\{c\rho^{\alpha-b}K^b,c\rho^{\alpha-b-1}K^{b+1}\}<\frac{1}{4}$ and $\delta=\frac{\rho}{2}$ we obtain $\|\Delta G(u)\|_{B(L^2;I)}\leq K$ and $\| G(u)\|_{B(I)}\leq \rho$, which means to say   $G(u)\in S_{\rho,K}$.

 To complete the proof we need to show that $G$ is a contraction on $S_{\rho,K}$. Repeating the above computations, it follows that 
\begin{equation}\label{Contrction1}
d(G(u),G(v))
\leq c \|\nabla \left(F(x,u)- F(x,v)\right)\|_{L^2_IL^{\frac{2N}{N+2}}_x}.    
\end{equation}
 Applying \eqref{SECONDEI} (with $\alpha>1$) we have
 \[
\begin{split}
\left|\nabla \left(F(x,u)-F(x,v)\right)\right|& \lesssim  |x|^{-b}(|u|^{\alpha}+|v|^{\alpha})||x|^{-1}(u-v)|+|x|^{-b}|u|^\alpha|\nabla (u- v)|\\
& \quad +|x|^{-b}\left(|u|^{\alpha-1}+|v|^{\alpha-1}\right)|\nabla v||u-v|,
\end{split}
\]
Combining \eqref{Contrction1} with Lemma \ref{LemCritico}, it follows that 
\[
\begin{split}
d(G(u),G(v))&\leq c\left( \|\Delta u\|^b_{B(L^2;I)} \|u\|^{\alpha-b}_{B(I)}+\|\Delta v\|^b_{B(L^2;I)} \|v\|^{\alpha-b}_{B(I)}\right) \|\Delta (u-v)\|_{B(L^2;I)}\\
&\quad +c\left( \|\Delta u\|^b_{B(L^2;I)} \|u\|^{\alpha-b-1}_{B(I)}+\|\Delta v\|^b_{B(L^2;I)} \|v\|^{\alpha-b-1}_{B(I)}\right) \|\Delta v\|_{B(L^2;I)}\|u-v\|_{B(I)}.
\end{split}
\]
Therefore, if $u, v \in S_{\rho,K}$ then
$$
d(G(u),G(v))\;\leq 2c (K^b\rho^{\alpha-b}+K^{b+1}\rho^{\alpha-b-1})d(u,v).
$$
Since the choice of $\rho$ and $K$ ove) implies
$$2c (K^b\rho^{\alpha-b}+K^{b+1}\rho^{\alpha-b-1})<1,
$$ 
we deduce that $G$ is a contraction on $S_{\rho,K}$.
So, by the contraction mapping principle, $G$ has a unique fixed point $u\in S_{\rho,K}$, which completes the proof of the theorem.
\end{proof}

Finally, we prove our stability or also called long time perturbation result for the energy critical inhomogeneous nonlinear Schr\"odinger equation. To this end, we first show a short time perturbation result, which will be used in the proof of Theorem \ref{LTPC}.
\begin{lemma}\label{STP}{\bf (Short-time perturbation).} 
Assume that assumptions in Theorem \ref{GWPCH2} hold.  Let $I\subseteq \mathbb{R}$ be a time interval containing zero and let $\widetilde{u}$ be as in Theorem \ref{erroreq} but with \eqref{erroreq1} replaced by
\begin{equation*}
\sup_{t\in I}  \|\tilde{u}\|_{\dot{H}^2_x}\leq M \;\; \textnormal{and}\;\; \|\tilde{u}\|_{B(I)}\leq \varepsilon,
\end{equation*}
for some positive constant $M$ and some small $\varepsilon>0$. Suppose that $u_0\in \dot{H}^2(\mathbb{R}^N)$ satisfies \eqref{erroreq2} for some $M'$.

\indent There exists $\varepsilon_0=\varepsilon_0(M,M')>0$ such that if $\varepsilon<\varepsilon_0$, then there is a unique solution $u$ of \eqref{IBNLS} on $I\times \mathbb{R}^N$ such that
\begin{equation}\label{C} 
\|u-\widetilde{u}\|_{B(I)}\lesssim \varepsilon 
\end{equation}
and
\begin{equation}\label{C1}
\|u\|_{B(I)}+\|\Delta u\|_{B(L^2; I)}\lesssim 1,
\end{equation}
where the implicit constants depend on $M$ and $M'$.
\end{lemma}
\begin{proof} 
Without loss of generality we may assume $0=\inf I$. First note by using $\|\widetilde{u}\|_{B(I)}\leq \varepsilon_0$, for some $\varepsilon_0>0$ enough small, \eqref{SE1} and \eqref{EstimativaImportante}, Lemma \ref{LemCritico} and a standard continuity argument we get $\|\Delta\widetilde{u}\|_{B(L^2;I)}\lesssim M.$

The idea is to obtain the solution $u$  as    $u=\widetilde{u}+w$, where $w$ is the solution of
\begin{equation}\label{IVPP} 
\begin{cases}
i\partial_tw +\Delta w + H(x,\widetilde{u},w)+e= 0,&  \\
w(0,x)= u_0(x)-\widetilde{u}_0(x),& 
\end{cases}
\end{equation}
with $H=H(x,\widetilde{u},w)=\lambda|x|^{-b} \left(|\widetilde{u}+w|^\alpha (\widetilde{u}+w)-|\widetilde{u}|^\alpha \widetilde{u}\right)$. So the main step is to show that \eqref{IVPP} has a solution. In order to apply the contraction mapping principle, let
\begin{equation*}
G (w)(t):=e^{it\Delta^2}w_0+i  \int_0^t e^{i(t-s)\Delta^2}(H(x,\widetilde{u},w)+e)(s)ds
\end{equation*}
and let $
B_{\rho,K}$ be the space of all functions $w$ defined on $I\times\mathbb{R}^N$ such that $\|w\|_{B(I)}\leq \rho$ and $\|\Delta w\|_{B(L^2;I)}\leq K$. Let us show that $G$ is a contraction on $B_{\rho,K}$.
As in Theorem \ref{GWPCH2} one has
\begin{equation}\label{ST1}
\begin{split}
\|G(w)\|_{B(I)}&\lesssim \varepsilon+ \| \nabla H\|_{L^2_IL_x^{\frac{2N}{N+2}}}+\|\nabla e\|_{L^2_IL_x^{\frac{2N}{N+2}}}
\end{split}
\end{equation}
and	
\begin{equation}\label{SP3}
\|\Delta G(w)\|_{B(L^2;I)}\lesssim  M'+ \| \nabla H\|_{L^2_IL_x^{\frac{2N}{N+2}}}+\|\nabla e\|_{L^2_IL_x^{\frac{2N}{N+2}}}.
\end{equation}	
Next we estimate $\|\nabla H\|_{L^2_IL_x^{\frac{2N}{N+2}}}$. From \eqref{SECONDEI} we obtain
\begin{equation*} 
\begin{split}
    |\nabla H|& \lesssim |x|^{-b}(|\widetilde{u}|^{\alpha}+|w|^{\alpha})|x|^{-1}|w|+|x|^{-b}(|\widetilde{u}|^\alpha+|w|^\alpha) |\nabla w|\\
    &\quad + |x|^{-b}\left(|\widetilde{u}|^{\alpha-1}+|w|^{\alpha-1}\right)|w||\nabla \widetilde{u}|.
\end{split}
\end{equation*}
Therefore, by using   Lemma \ref{LemCritico} we deduce
\[
\begin{split}
    \|\nabla H\|_{L^2_IL_x^{\frac{2N}{N-2}}}& \lesssim \left(\|\Delta \widetilde{u} \|^b_{B(L^2;I)}\| \widetilde{u} \|^{\alpha-b}_{B(I)} + \|\Delta w \|^b_{B(L^2;I)}\| w\|^{\alpha-b}_{B(I)} \right)\|\Delta w \|_{B(L^2;I)}\\
    &\quad +\left(\|\Delta \widetilde{u} \|^b_{B(L^2;I)} \| \widetilde{u} \|^{\alpha-1-b}_{B(I)} + \|\Delta w \|^b_{B(L^2;I)} \| w \|^{\alpha-1-b}_{B(I)} \right) \| w \|_{B(I)}  \|\Delta \widetilde{u} \|_{B(L^2;I)}.
\end{split}
\]
Gathering together the above estimates with our assumptions, it follows that for any $w\in B_{\rho,K}$,
\begin{align}\label{SP9}
\|\nabla H\|_{L^2_IL_x^{\frac{2N}{N+2}}} \lesssim \left(M^b\varepsilon^{\alpha-b}+K^b \rho^{\alpha-b}\right)K +\left( M^b \varepsilon^{\alpha-1-\theta} + K^b \rho^{\alpha-1-b} \right) \rho M.
\end{align}
Thus, by choosing $\rho=2c\varepsilon$ and $K=3cM'$, the relations \eqref{ST1}, \eqref{SP3} and \eqref{SP9} yield 
$$
\|G(w)\|_{B(I)}\leq  \tfrac{\rho}{2}+ c\left(M^bK+K^{b+1}+M^{b+1}+K^bM \right)\rho^{\alpha-b}=:\tfrac{\rho}{2}+ cA\rho^{\alpha-b},
$$
\begin{equation*}
\|\Delta G(w)\|_{B(L^2;I)}\leq \tfrac{K}{3}+c(M^b+K^b)\rho^{\alpha-b}K+c(\varepsilon+M^{b+1}+K^bM)\rho^{\alpha-b}.
\end{equation*}
Setting $\varepsilon_0$ sufficiently small such that 
$$
cA\rho^{\alpha-b-1}<\tfrac{1}{2},\;\;\;c(M^b+K^b)\rho^{\alpha-b}<\tfrac{1}{3}\;\;\textnormal{and}\;\;c(\varepsilon+M^{b+1}+K^bM)\rho^{\alpha-b}<\tfrac{K}{3},
$$
we obtain
\begin{equation*}
\|G(w)\|_{B(I)}\leq \rho\;\;\;\textnormal{and}\;\;\;\|\Delta G(w)\|_{B(L^2;I)}\leq K,
\end{equation*}
that is, $G$ is well defined on $B_{\rho,K}$. By using a similar argument we can also show that $G$ is a contraction. Thus, we have a unique solution $w$ on $I\times \mathbb{R}^N$ such that 
$$
\|w\|_{B(I)}\lesssim \varepsilon \;\;\;\textnormal{and}\;\;\;\|\Delta w\|_{B(L^2;I)} \lesssim 1,
$$ 
which it turn implies \eqref{C} and \eqref{C1}. This completes the proof of the theorem.
\end{proof}	

With Lemma \ref{STP} in hand we are able to show Theorem \ref{LTPC}.

\begin{proof}[\bf {Proof of Theorem \ref{LTPC}}] Without loss of generality we may assume $0=\inf I$. We first split the interval $I$ into $n = n(L,\varepsilon)$ intervals, say, $I_j = [t_j ,t_{j+1}]$ such that $\|\widetilde{u}\|_{B(I)}\leq \varepsilon$, where  $\varepsilon<\varepsilon_0(M,2M')$ and $\varepsilon_0$ is given in Lemma \ref{STP}; this is always possible because  $\|\widetilde{u}\|_{B(I)}\leq L$.

On each interval $I_j$ we have that
\begin{equation*}
w(t)=e^{i(t-t_j)\Delta}w(t_j)+i\int_{t_j}^{t}e^{i(t-s)\Delta}(H(x,\widetilde{u},w)+e)(s)ds,
  \end{equation*}
solves \eqref{IVPP} with initial data $w(t_j)=u(t_j)-\widetilde{u}(t_j)$. Assume for the moment that
\begin{equation}\label{LP3}
 \|e^{i(t-t_j)\Delta^2}(u(t_j)-\widetilde{u}(t_j))\|_{B(I_j)}\leq c(M,M',j)\varepsilon\leq \varepsilon_0
 \end{equation}
 and
 \begin{equation}\label{LP4}
 \|u(t_j)-\widetilde{u}(t_j)\|_{\dot{H}^2_x}\leq 2M'.
 \end{equation}
If $\varepsilon_1$ is small enough we may reiterate Lemma \ref{STP} to get  (for $\varepsilon<\varepsilon_1$), 
\begin{equation}\label{LP1}
\|u-\widetilde{u}\|_{B(I_j)}\lesssim \varepsilon
\end{equation}
and
\begin{equation}\label{LP2}
\|\Delta w\|_{B(L^2;I_j)}\lesssim 1.
\end{equation}
In view of \eqref{LP1} and \eqref{LP2}, a summation over all intervals $I_j$ yields the desired. 
 
In order to complete the proof, it suffices to establish \eqref{LP3} and \eqref{LP4}. As in \eqref{trii}, we obtain
\[
\begin{split}
 \|e^{i(t-t_j)\Delta^2}w(t_j)\|_{B(I_j)}&\lesssim \|e^{it\Delta}w_0\|_{B(I_j)}+\| \nabla H\|_{L^2_{[0,t_j]}L_x^{\frac{2N}{N+2}}}+\|\nabla e\|_{L^2_{I}L_x^{\frac{2N}{N+2}}}.
\end{split}
 \]
Recall from \eqref{SP9} and the choice of $\rho$ in Lemma \ref{STP} that $\|\nabla H\|_{L^2_{[0,t_j]}L_x^{\frac{2N}{N+2}}}\leq c(M,M')\varepsilon^{\alpha-b}$. Therefore, after an inductive argument, we infer
$$
\|e^{i(t-t_j)\Delta^2}(u(t_j)-\widetilde{u}(t_j))\|_{B(I_j)}\lesssim \varepsilon+\sum_{k=0}^{j-1}c(M,M',k)\varepsilon^{\alpha-b}.
$$
 In a similar fashion,
\[
\begin{split}
\|u(t_j)-\widetilde{u}(t_j)\|_{\dot{H}^2_x}&\lesssim  
M'+\sum_{k=0}^{j-1}C(k,M,M')\varepsilon^{\alpha-\theta}+\varepsilon.
\end{split}
\]
Consequently, \eqref{LP3} and \eqref{LP4} hold provided  $\varepsilon_1$ is chosen to be small enough. The proof of the theorem is thus completed. 
\end{proof}

\section*{Acknowledgments} 
A.P. was partially supported by  CNPq/Brazil grant 303762/2019-5 and FAPESP/Brazil grant 2019/02512-5.

\bibliographystyle{abbrv}
\bibliography{bibguzman}	

\end{document}